\documentclass[12pt, reqno]{amsart}
\usepackage{amsmath, amsthm, amscd, amsfonts, amssymb, graphicx}
\usepackage[bookmarksnumbered, colorlinks, plainpages]{hyperref}
\input{mathrsfs.sty}

\newtheorem{theorem}{Theorem}[section]
\newtheorem{lemma}[theorem]{Lemma}

\newtheorem{corollary}[theorem]{Corollary}
\theoremstyle{definition}

\theoremstyle{remark}
\newtheorem{remark}[theorem]{Remark}

\numberwithin{equation}{section}

\begin{document}
\title[Operator parallelogram law]{An operator extension of the parallelogram law and related norm inequalities}
\author[M.S. Moslehian]{Mohammad Sal Moslehian}
\address{Department of Mathematics, Center of Excellence in Analysis on Algebraic Structures (CEAAS), Ferdowsi University of Mashhad, P.O. Box 1159, Mashhad 91775, Iran.}
\email{moslehian@ferdowsi.um.ac.ir and moslehian@ams.org}

\dedicatory{Dedicated to my teacher Aziz Atai Langroudi with respect and affection}



\begin{abstract} We establish a general operator parallelogram law
concerning a characterization of inner product spaces, get an
operator extension of Bohr's inequality and present several norm
inequalities. More precisely, let ${\mathfrak A}$
be a $C^*$-algebra, $T$ be a locally compact Hausdorff space
equipped with a Radon measure $\mu$ and let $(A_t)_{t\in T}$ be a
continuous field of operators in ${\mathfrak A}$ such that the
function $t \mapsto A_t$ is norm continuous on $T$ and the function
$t \mapsto \|A_t\|$ is integrable. If $\alpha: T \times T \to
\mathbb{C}$ is a measurable function such that $\overline
{\alpha(t,s)}\alpha(s,t)=1$ for all $t, s \in T$, then we show that
\begin{align*}
\int_T\int_T&\left|\alpha(t,s) A_t-\alpha(s,t) A_s\right|^2d\mu(t)d\mu(s)+\int_T\int_T\left|\alpha(t,s) B_t-\alpha(s,t) B_s\right|^2d\mu(t)d\mu(s) \nonumber\\
&= 2\int_T\int_T\left|\alpha(t,s) A_t-\alpha(s,t)
B_s\right|^2d\mu(t)d\mu(s) - 2\left|\int_T(A_t-B_t)d\mu(t)\right|^2\,.
\end{align*}

\end{abstract}
\maketitle


\section{Introduction}

Let ${\mathfrak A}$ be a $C^*$-algebra and let $T$ be a locally
compact Hausdorff space. A field $(A_t)_{t\in T}$ of operators in
${\mathfrak A}$ is called a continuous field of operators if the
function $t \mapsto A_t$ is norm continuous on $T$. If $\mu(t)$ is a
Radon measure on $T$ and the function $t \mapsto \|A_t\|$ is
integrable, one can form the Bochner integral $\int_{T}A_t{\rm
d}\mu(t)$, which is the unique element in ${\mathfrak A}$ such that
$$\varphi\left(\int_TA_t{\rm d}\mu(t)\right)=\int_T\varphi(A_t){\rm d}\mu(t)$$
for every linear functional $\varphi$ in the norm dual ${\mathfrak
A}^*$ of ${\mathfrak A}$; see \cite[Section 4.1]{H-P} and \cite{H-P-P}.

Let ${\mathbb B}({\mathscr H})$ be the algebra of all bounded linear
operators on a separable complex Hilbert space ${\mathscr H}$
endowed with inner product $\langle \cdot , \cdot \rangle $. We denote the absolute value of $A
\in {\mathbb B}({\mathscr H})$ by $|A|=(A^*A)^{1/2}$. For $x, y \in
{\mathscr H}$, the rank one operator $x \otimes y$ is defined on
${\mathscr H}$ by $(x\otimes y)(z)=\langle z,y\rangle x$.

\noindent Let $A\in \mathbb{B}(\mathscr{H})$ be a compact operator
and let $0 < p < \infty$. The Schatten $p$-norm ($p$-quasi-norm) for
$1 \leq p < \infty \,(0< p < 1)$ is defined by $\|A\|_p=(\textrm{tr}
|A|^p)^{1/p}$, where $\textrm{tr}$ is the usual trace functional.
Clearly
\begin{eqnarray}\label{trace}
\left\|\,|A|^p\,\right\|_1=\|A\|_p^p
\end{eqnarray}
for $p>0$. For $p>0$, the Schatten $p$-class, denoted by
$\mathcal{C}_p$, is defined to be the two-sided ideal in
$\mathbb{B}(\mathscr{H})$ of those compact operators $A$ for which
$\|A\|_p$ is finite. In particular, $\mathcal{C}_1$ and
$\mathcal{C}_2$ are the trace class and the Hilbert-Schmidt class,
respectively. For $1 \leq p <\infty$, $\mathcal{C}_p$ is a Banach
space; in particular the triangle inequality holds. However, for
$0<p<1$, the quasi-norm $\|.\|_p$ does not satisfy the triangle
inequality. In addition to Schatten $p$-norms and operator norm,
there are other interesting norms defined on some ideals contained
in the ideal of compact operators. A unitarily invariant norm
$|||\cdot|||$ is defined only on a norm ideal
$\mathcal{C}_{|||\cdot|||}$ associated with it and has the property
$|||UAV|||=|||A|||$, where $U$ and $V$ are unitaries and $A
\in\mathcal{C}_{|||\cdot|||}$. For more information on the theory of
the unitarily invariant norms the reader is referred to \cite{SIM}.

\noindent One can show that
\begin{eqnarray}\label{00}
\sum_{i,j=1}^n\|x_i-x_j\|^2+ \sum_{i,j=1}^n\|y_i-y_j\|^2 =
2\sum_{i,j=1}^n\|x_i-y_j\|^2 -2\left\|\sum_{i=1}^n
(x_i-y_i)\right\|^2
\end{eqnarray}
holds in an inner product space, which is indeed a generalization of the classical \emph{parallelogram law}:
$$|z+w|^2+|z-w|^2=2|z|^2+2|w|^2 \qquad (z, w \in \mathbb{C})\,.$$
There are several extensions of parallelogram law among them we
could refer the interested reader to \cite{C-C-M-P, EGE, KAT, ZEN,
Z-F}. Generalizations of the parallelogram law for the Schatten
$p$-norms have been given in the form of the celebrated Clarkson
inequalities (see \cite{H-K} and references therein). Since
$\mathcal{C}_2$ is a Hilbert space under the inner product $\langle
A, B\rangle=\textrm{tr} (B^*A)$, it follows from \eqref{00} that if
$A_1, \cdots, A_n, B_1, \cdots, B_n \in \mathcal{C}_2$ with
$\sum_{i=1}^n(A_i-B_i)=0$, then
\begin{eqnarray}\label{4}
\sum_{i,j=1}^n\|A_i-A_j\|_2^2+\sum_{i,j=1}^n\|B_i-B_j\|_2^2 =
2\sum_{i,j=1}^n\|A_i-B_j\|_2^2\,.
\end{eqnarray}

The classical \emph{Bohr's inequality} states that for any $z, w \in
{\mathbb C}$ and any positive real numbers $r, s$ with
$\frac{1}{r}+\frac{1}{s}=1$,
$$|z+w|^2 \leq r |z|^2 + s |w|^2.$$
\noindent Many interesting operator generalizations of this
inequality have been obtained; cf. \cite{A-B-P, C-P, HIR, M-P-P, ZHA}.

In this paper, we establish an extended operator parallelogram law and get a generalization of Bohr's inequality.
We also present several unitarily invariant and Schatten $p$-norm inequalities. Our results can be
regarded as extensions of main results of \cite{H-K-M}.


\section{Joint extensions of the parallelogram law and Bohr's inequality}

We start our work with following clear lemma. The first equality is called
the operator parallelogram law.

\begin{lemma}\label{lem1}
Let $A, B \in \mathbb{B}(\mathscr{H})$. Then
\begin{eqnarray*}
|A+B|^2+|A-B|^2=2|A|^2+2|B|^2\qquad (\textrm{a});
\end{eqnarray*}and
\begin{eqnarray*}
|A+B|^2-|A-B|^2=4\textrm{Re}(A^*B)\qquad (\textrm{b})
\end{eqnarray*}
\end{lemma}


We now state our main result, which is an operator version of equality \eqref{00}. As we will see later, it is indeed a joint operator extension of Bohr and parallelogram inequalities.

\begin{theorem}\label{th1}
Let ${\mathfrak A}$ be a $C^*$-algebra, $T$ be a locally compact
Hausdorff space equipped with a Radon measure $\mu$ and let
$(A_t)_{t\in T}$ be a continuous field of operators in ${\mathfrak
A}$ such that the function $t \mapsto A_t$ is norm continuous on $T$
and the function $t \mapsto \|A_t\|$ is integrable. Let $\alpha: T \times T \to \mathbb{C}$ be a measurable function such that $\overline {\alpha(t,s)}\alpha(s,t)=1$ for all $t, s \in T$. Then
\begin{align*}
\int_T\int_T&\left|\alpha(t,s) A_t-\alpha(s,t) A_s\right|^2d\mu(t)d\mu(s)+\int_T\int_T\left|\alpha(t,s) B_t-\alpha(s,t) B_s\right|^2d\mu(t)d\mu(s) \nonumber\\
&= 2\int_T\int_T\left|\alpha(t,s) A_t-\alpha(s,t)
B_s\right|^2d\mu(t)d\mu(s) - 2\left|\int_T(A_t-B_t)d\mu(t)\right|^2\,.
\end{align*}
\end{theorem}
\begin{proof}
\begin{align*}
\int_T&\int_T\big|\alpha(t,s) A_t-\alpha(s,t) A_s\big|^2d\mu(t)d\mu(s)+\int_T\int_T\big|\alpha(t,s) B_t-\alpha(s,t) B_s\big|^2d\mu(t)d\mu(s)\\
&=\int_T\int_T\big(\big|\alpha(t,s) A_t-\alpha(s,t) A_s\big|^2+\big|\alpha(t,s) B_t-\alpha(s,t) B_s\big|^2\big)d\mu(t)d\mu(s)\\
&=\int_T\int_T\Big(\frac{1}{2}\big|\alpha(t,s) A_t-\alpha(s,t) A_s + \alpha(t,s) B_t-\alpha(s,t) B_s\big|^2\\
&\quad + \frac{1}{2}\big| \alpha(t,s) A_t-\alpha(s,t) A_s - \alpha(t,s) B_t+\alpha(s,t) B_s \big|^2\Big)d\mu(t)d\mu(s)\\
&\qquad\qquad \qquad\qquad \qquad\qquad\qquad\qquad \qquad\qquad\qquad \qquad \qquad (\textrm{by Lemma\, \ref{lem1}(a)})\\
&=\int_T\int_T \big(\frac{1}{2}\big|(\alpha(t,s) A_t-\alpha(s,t) B_s)-(\alpha(s,t) A_s-\alpha(t,s) B_t)\big|^2 \\
&\quad + \frac{1}{2}\big|(\alpha(t,s) A_t-\alpha(t,s) B_t)-(\alpha(s,t) A_s-\alpha(s,t) B_s)\big|^2\Big)d\mu(t)d\mu(s)\\
&=\int_T\int_T \Big[\Big(\big|\alpha(t,s) A_t-\alpha(s,t) B_s\big|^2+\big|\alpha(s,t) A_s-\alpha(t,s) B_t\big|^2\\
& \quad -\frac{1}{2}\big|(\alpha(t,s) A_t-\alpha(s,t) B_s)+(\alpha(s,t) A_s-\alpha(t,s) B_t)\big|^2\Big)\\
&\quad +\frac{1}{2}\big|(\alpha(t,s) A_t-\alpha(t,s) B_t)-(\alpha(s,t) A_s-\alpha(s,t) B_s)\big|^2\Big]d\mu(t)d\mu(s)\\
& \qquad\qquad \qquad\qquad\qquad\qquad \qquad\qquad \qquad\qquad \qquad\qquad \qquad (\textrm{by Lemma\, \ref{lem1}(a)})\\
&= \int_T\int_T\big|\alpha(t,s) A_t-\alpha(s,t) B_s\big|^2d\mu(t)d\mu(s)+ \int_T\int_T\big|\alpha(t,s) A_t-\alpha(s,t) B_s\big|^2d\mu(t)d\mu(s)\\
&\quad -\frac{1}{2}\int_T\int_T\Big(\big|(\alpha(t,s) A_t-\alpha(t,s) B_t)+(\alpha(s,t) A_s-\alpha(s,t) B_s)\big|^2 \\
& \quad -\big|(\alpha(t,s) A_t-\alpha(t,s) B_t)-(\alpha(s,t) A_s-\alpha(s,t) B_s)\big|^2\Big)d\mu(t)d\mu(s)\\
&= 2\int_T\int_T\big|\alpha(t,s) A_t-\alpha(s,t) B_s\big|^2d\mu(t)d\mu(s)\\
& \quad -2\,\textrm{Re}\int_T\int_T(\alpha(t,s) A_t-\alpha(t,s) B_t)^*(\alpha(s,t) A_s-\alpha(s,t) B_s)d\mu(t)d\mu(s)\\
&\qquad\qquad \qquad\qquad \qquad\qquad\qquad\qquad \qquad\qquad\qquad \qquad \qquad (\textrm{by Lemma\, \ref{lem1}(b)})\\
&=2\int_T\int_T\big|\alpha(t,s) A_t-\alpha(s,t) B_s\big|^2d\mu(t)d\mu(s)\\
&\quad -2\,\textrm{Re}\Big[\Big(\int_T(A_t-B_t)d\mu(t)\Big)^*\Big(\int_T (A_s-B_s)d\mu(s)\Big)\Big]\\
&=2\int_T\int_T\big|\alpha(t,s) A_t-\alpha(s,t) B_s\big|^2d\mu(t)d\mu(s) - 2\big|\int_T( A_t- B_t)d\mu(t)\big |^2\,.
\end{align*}
\end{proof}

If we let $T=\{1, \cdots, n\}$, $\mu$ be the counting measure on $T$
and $\alpha(i,j)=\sqrt{\frac{r_i}{r_j}}$, where $r_i> 0\,\,(1 \leq i
\leq n)$, in Theorem \ref{th1}, then we get
\begin{corollary}[Generalized Parallelogram Law]\label{cor0}
Let $A_1, \cdots, A_n, B_1, \cdots, B_n \in \mathbb{B}(\mathscr{H})$
and let $r_1, \cdots, r_n$ be positive numbers. Then
\begin{eqnarray}\label{oit}
\sum_{1\leq i<j\leq n}\left|\sqrt{\frac{r_i}{r_j}}A_i-\sqrt{\frac{r_j}{r_i}}A_j\right|^2&+&\sum_{1\leq i<j\leq n}\left|\sqrt{\frac{r_i}{r_j}}B_i-\sqrt{\frac{r_j}{r_i}}B_j\right|^2 \nonumber\\
&=&\sum_{i,j=1}^n\left|\sqrt{\frac{r_i}{r_j}}A_i-\sqrt{\frac{r_j}{r_i}}B_j\right|^2
- \left|\sum_{i=1}^n(A_i-B_i)\right|^2\,.
\end{eqnarray}
\end{corollary}

If we set $B_1=\cdots=B_n=0$ in Corollary \ref{cor0}, then the following extension of parallelogram law is obtained
\begin{corollary} \cite[Theorem 4.2]{Z-F}
Suppose that $A_1, \cdots, A_n \in \mathbb{B}(\mathscr{H})$ and
$r_1, \cdots, r_n$ are positive numbers with
$\sum_{i=1}^n\frac{1}{r_i}=1$. Then
\begin{eqnarray}\label{zf}
(0 \leq) \sum_{1 \leq i<j\leq
n}\left|\sqrt{\frac{r_i}{r_j}}A_i-\sqrt{\frac{r_j}{r_i}}A_j\right|^2=
\sum_{i=1}^n r_i\left|A_i\right|^2
-\left|\sum_{i=1}^nA_i\right|^2\,.
\end{eqnarray}
\end{corollary}
\begin{remark}
If $n=2$ and $t:=\frac{r_1}{r_2}>0$, then operator equality \eqref{zf} can be restated as the following form which is, as noted in \cite{Z-F}, a generalization of \cite[Theorem 1]{HIR} and \cite[Theorem 3]{C-P}.
$$|A_1+A_2|^2+\frac{1}{t}|tA_1-A_2|^2=(1+t)|A_1|^2+(1+\frac{1}{t})|A_2|^2\,.$$
\end{remark}
We also infer the following extension of Bohr's inequality \cite[Theorem 7]{ZHA} from \eqref{zf}.
$$\left|\sum_{i=1}^nA_i\right|^2 \leq \sum_{i=1}^n r_i \left|A_i\right|^2\,.$$

The preceding inequality is, in turn, a special case of the following general form of the Bohr inequality being easily deduced from the fact that the left hand side of the operator equality in Theorem \ref{th1} is a positive element of the $C^*$-algebra ${\mathfrak A}$.

\begin{corollary}[Generalized Operator Bohr's Inequality]
Let ${\mathfrak A}$ be a $C^*$-algebra, $T$ be a locally compact
Hausdorff space equipped with a Radon measure $\mu$ and let
$(A_t)_{t\in T}$ be a continuous field of operators in ${\mathfrak
A}$ such that the function $t \mapsto A_t$ is norm continuous on $T$
and the function $t \mapsto \|A_t\|$ is integrable. Let $\alpha: T \times T \to \mathbb{C}$ be a measurable function such that $\overline {\alpha(t,s)}\alpha(s,t)=1$ for all $t, s \in T$. Then
\begin{align*}
\big|\int_T(A_t-B_t)d\mu(t)\big|^2 \leq \int_T\int_T\big|\alpha(t,s) A_t-\alpha(s,t) B_s\big|^2d\mu(t)d\mu(s)\,.
\end{align*}
\end{corollary}


A weighted extension of norm equality \eqref{00} can be deduced from \eqref{oit} as follows:

\begin{corollary}
Let $x_1, \cdots, x_n, y_1, \cdots, y_n\in {\mathscr H}$ and let
$r_1, \cdots, r_n$ be positive numbers. Then
\begin{eqnarray*}
\sum_{1 \leq i < j \leq n}\left\|\sqrt{\frac{r_i}{r_j}}x_i-\sqrt{\frac{r_j}{r_i}}x_j\right\|^2 &+& \sum_{1 \leq i < j \leq n}\left\|\sqrt{\frac{r_i}{r_j}}y_i-\sqrt{\frac{r_j}{r_i}}y_j\right\|^2 \\&=&
\sum_{i,j=1}^n\left\|\sqrt{\frac{r_i}{r_j}}x_i-\sqrt{\frac{r_j}{r_i}}y_j\right\|^2 -\left\|\sum_{i=1}^n
(x_i-y_i)\right\|^2\,.
\end{eqnarray*}
\end{corollary}
\begin{proof}
Let $e$ be a non-zero vector of ${\mathscr H}$ and set $A_i=x_i
\otimes e, B_i=y_i \otimes e$ for $i=1, \cdots, n$. It follows from
the elementary properties of rank one operators and equality
\eqref{oit} that
\begin{eqnarray*}
&&\hspace{-1in}\Big(\sum_{i,j=1}^n\left\|\sqrt{\frac{r_i}{r_j}}x_i-\sqrt{\frac{r_j}{r_i}}x_j\right\|^2+
\sum_{i,j=1}^n\left\|\sqrt{\frac{r_i}{r_j}}y_i-\sqrt{\frac{r_j}{r_i}}y_j\right\|^2\Big) e\otimes
e\\
&=& \sum_{i,j=1}^n \left|\left(\sqrt{\frac{r_i}{r_j}}x_i-\sqrt{\frac{r_j}{r_i}}x_j\right)\otimes e\right|^2 + \sum_{i,j=1}^n \left|\left(\sqrt{\frac{r_i}{r_j}}y_i-\sqrt{\frac{r_j}{r_i}}y_j\right)\otimes e\right|^2 \\
&=&\sum_{i,j=1}^n\left|\sqrt{\frac{r_i}{r_j}}A_i-\sqrt{\frac{r_j}{r_i}}A_j\right|^2+\sum_{i,j=1}^n\left|\sqrt{\frac{r_i}{r_j}}B_i-\sqrt{\frac{r_j}{r_i}}B_j\right|^2\\
&=& 2\sum_{i,j=1}^n\left|\sqrt{\frac{r_i}{r_j}}A_i-\sqrt{\frac{r_j}{r_i}}B_j\right|^2 -
2\left|\sum_{i=1}^n(A_i-B_i)\right|^2\\
&=&2\sum_{i,j=1}^n \left|\left(\sqrt{\frac{r_i}{r_j}}x_i-\sqrt{\frac{r_j}{r_i}}y_j\right)\otimes e\right|^2 -2 \left|\sum_{i=1}^n (x_i-y_i) \otimes e\right|^2\,.\\
&=& \left(2\sum_{i,j=1}^n\left\|\sqrt{\frac{r_i}{r_j}}x_i-\sqrt{\frac{r_j}{r_i}}y_j\right\|^2 -2\left\|\sum_{i=1}^n
(x_i-y_i)\right\|^2\right)e\otimes e\,,
\end{eqnarray*}
from which we conclude the result.
\end{proof}


\section{A general parallelogram law}

\noindent An extension of a result of \cite{A-S} for $n$-tuples may
be stated as follows.

$\bullet$ If $f: [0,\infty) \to [0,\infty)$ is a convex function,
then for any positive operators $A_1, \cdots, A_n$, any nonnegative
numbers $\alpha_1, \cdots, \alpha_n$ with $\sum_{j=1}^n\alpha_j=1$
and any unitarily invariant norm $|||\cdot|||$ on ${\mathbb
B}({\mathscr H})$
\begin{eqnarray}\label{4.1}
\left|\left|\left|\sum_{j=1}^n \alpha_jf(A_j)\right|\right|\right|
\geq \left|\left|\left|f\left(\sum_{j=1}^n
\alpha_jA_j\right)\right|\right|\right|\,.
\end{eqnarray}

\noindent The following result is also known \cite{KOS}:

$\bullet$ If $f: [0,\infty) \to [0,\infty)$ is a convex function
with $f(0)=0$, then for any positive operators $A_1, \cdots, A_n$
and any unitarily invariant norm $|||\cdot|||$ on ${\mathbb
B}({\mathscr H})$
\begin{eqnarray}\label{4.2}
\left|\left|\left|f\left(\sum_{j=1}^n
A_j\right)\right|\right|\right| \geq \left|\left|\left|\sum_{j=1}^n
f(A_j)\right|\right|\right| \,.
\end{eqnarray}

The reverse inequalities hold for concave functions; see \cite{H-K}
for more details. We now prove another significant theorem.

\begin{theorem}\label{main}
Let $A_1, \cdots, A_n \in \mathcal{C}_{|||\cdot|||}$, $r_1, \cdots,
r_n$ be positive real numbers with $\sum_{i=1}^n\frac{1}{r_i}=1$,
let $g$ be a nonnegative convex function on $[0,\infty)$ such that
$g(0)=0$ and let $f(t)=g(t^2)$. Then
\begin{eqnarray}\label{4.3}
\left|\left|\left|\sum_{i=1}^n\frac{1}{r_i}f\left(|r_iA_i|\right)\right|\right|\right|
\geq \left|\left|\left|\sum_{1 \leq i<j\leq
n}f\left(\left|\sqrt{\frac{r_i}
{r_j}}A_i-\sqrt{\frac{r_j}{r_i}}A_j\right|\right)+
f\left(\left|\sum_{i=1}^nA_i\right|\right)\right|\right|\right|
\end{eqnarray}
for all unitarily invariant norm. If $g$ is concave, then the
reverse of \eqref{4.3} holds.
\end{theorem}
\begin{proof}
\begin{eqnarray*}
\left|\left|\left|\sum_{i=1}^n\frac{1}{r_i}f\left(|r_iA_i|\right)\right|\right|\right|
&=&\left|\left|\left|\sum_{i=1}^n\frac{1}{r_i}g\left(|r_iA_i|^2\right)\right|\right|\right|\\
&\geq&\left|\left|\left|g\left(\sum_{i=1}^n\frac{1}{r_i}|r_iA_i|^2\right)\right|\right|\right| \quad\quad\qquad\qquad\qquad\qquad\qquad\quad\ (\textrm{by~} \eqref{4.1})\\
&=& \left|\left|\left|g\left(\sum_{1 \leq i<j\leq
n}\left|\sqrt{\frac{r_i}
{r_j}}A_i-\sqrt{\frac{r_j}{r_i}}A_j\right|^2+\left|\sum_{i=1}^nA_i\right|^2\right)\right|\right|\right| \quad\quad(\textrm{by~} \eqref{zf})\\
&\geq& \left|\left|\left|\sum_{1 \leq i<j\leq
n}g\left(\left|\sqrt{\frac{r_i}
{r_j}}A_i-\sqrt{\frac{r_j}{r_i}}A_j\right|^2\right)+g\left(\left|\sum_{i=1}^nA_i\right|^2\right)\right|\right|\right| (\textrm{by~} \eqref{4.2})\\
&=& \left|\left|\left|\sum_{1 \leq i<j\leq
n}f\left(\left|\sqrt{\frac{r_i}
{r_j}}A_i-\sqrt{\frac{r_j}{r_i}}A_j\right|\right)+f\left(\left|\sum_{i=1}^nA_i\right|\right)\right|\right|\right|\,.
\end{eqnarray*}
\end{proof}

The function $g(t)=t^p$ for $1 \leq p < \infty$ ($g(t)=t^p$ for $0 <
p \leq 1 $, resp.) is convex (concave, resp.) on $[0,\infty)$. Hence
we get the following corollary.

\begin{corollary}
Let $A_1, \cdots, A_n \in \mathcal{C}_p$ and $r_1, \cdots, r_n$ be
positive real numbers with $\sum_{i=1}^n\frac{1}{r_i}=1$. Then
\begin{eqnarray*}
\sum_{i=1}^n r_i^{p-1}\|A_i\|_p^p\geq \sum_{1 \leq i<j\leq n}
\left\|\sqrt{\frac{r_i}
{r_j}}A_i-\sqrt{\frac{r_j}{r_i}}A_j\right\|_p^p +
\left\|\sum_{i=1}^nA_i\right\|_p^p
\end{eqnarray*}
for any $2 \leq p <\infty$. The reverse inequality holds for any $0
< p \leq 2 $.
\end{corollary}
\begin{proof} Let $2 \leq p <\infty$.
\begin{eqnarray*}
\sum_{i=1}^n r_i^{p-1}\|A_i\|_p^p &=& \sum_{i=1}^n
r_i^{p-1}\|\,|A_i|^p\|_1 \qquad\qquad\qquad\qquad\qquad\qquad\qquad\qquad\qquad (\textrm{by~} \eqref{trace})\\
&=&{\rm tr}(\sum_{i=1}^n r_i^{p-1}|A_i|^p)\\
&=&\left\|\sum_{i=1}^n \frac{1}{r_i}\left|r_iA_i\right|^p\right\|_1\\
&\geq&\left \|\sum_{1 \leq i<j\leq n} \left|\sqrt{\frac{r_i}
{r_j}}A_i-\sqrt{\frac{r_j}{r_i}}A_j\right|^p +
\left|\sum_{i=1}^nA_i\right|^p\right\|_1\\
&&\qquad\qquad\qquad\qquad\qquad (\textrm{by~ Theorem~}\ref{main} \textrm{~for~} g(t)=t^{\frac{p}{2}};\, 2 \leq p < \infty)\\\\
&=&{\rm tr}\left(\sum_{1 \leq i<j\leq n} \left|\sqrt{\frac{r_i}
{r_j}}A_i-\sqrt{\frac{r_j}{r_i}}A_j\right|^p +
\left|\sum_{i=1}^nA_i\right|^p\right)\\
&=&\sum_{1 \leq i<j\leq n}{\rm tr}\left(\left|\sqrt{\frac{r_i}
{r_j}}A_i-\sqrt{\frac{r_j}{r_i}}A_j\right|^p\right) + {\rm tr}\left(
\left|\sum_{i=1}^nA_i\right|^p\right)\\
&=&\sum_{1 \leq i<j\leq n} \left \|\,\left|\sqrt{\frac{r_i}
{r_j}}A_i-\sqrt{\frac{r_j}{r_i}}A_j\right|^p \right\|_1+\left\|\,
\left|\sum_{i=1}^nA_i\right|^p\right\|_1\\
&=&\sum_{1 \leq i<j\leq n} \left\|\sqrt{\frac{r_i}
{r_j}}A_i-\sqrt{\frac{r_j}{r_i}}A_j\right\|^p_p+ \left\|\sum_{i=1}^nA_i\right\|_p^p\,.\qquad\quad\qquad\quad (\textrm{by~} \eqref{trace})\\\\
\end{eqnarray*}
The proof for the reverse inequality is similar.
\end{proof}

\end{document}